\documentclass{amsart}
\usepackage{amssymb}
\usepackage{amsmath}
\usepackage{amsfonts}

\newtheorem{theorem}{Theorem}
\theoremstyle{plain}

\newtheorem{corollary}{Corollary}

\newtheorem{definition}{Definition}
\newtheorem{example}{Example}

\newtheorem{lemma}{Lemma}

\newtheorem{remark}{Remark}

\numberwithin{equation}{section}
\input{tcilatex}

\begin{document}
\title{ANTI-INVARIANT RIEMANNIAN SUBMERSIONS FROM COSYMPLECTIC MANIFOLDS}
\author{C. Murathan}
\address{Art and Science Faculty,Department of Mathematics, Uludag
University, 16059 Bursa, TURKEY}
\email{cengiz@uludag.edu.tr}
\author{I. K\"{u}peli Erken}
\address{Art and Science Faculty,Department of Mathematics, Uludag
University, 16059 Bursa, TURKEY}
\email{iremkupeli@uludag.edu.tr}
\date{20.02.2013}
\keywords{Riemannian submersion, Cosymplectic manifold, Anti-invariant
submersion This paper is supported by Uludag University research project
(KUAP(F)-2012/57).}

\begin{abstract}
We introduce anti-invariant Riemannian submersions from cosymplectic
manifolds onto Riemannian manifolds. We survey main results of
anti-invariant Riemannian submersions defined on cosymplectic manifolds. We
investigate necessary and sufficient condition for an anti-invariant
Riemannian submersion to be totally geodesic and harmonic. We give examples
of anti-invariant submersions such that characteristic vector field $\xi $
is vertical or horizontal. Moreover we give decomposition theorems by using
the existence of anti-invariant Riemannian submersions.
\end{abstract}

\maketitle

\section{\textbf{Introduction}}

In \cite{BO1}, O'Neill defined a Riemannian submersion, which is the
\textquotedblleft dual\textquotedblright\ notion of isometric immersion, and
obtained some fundamental equations corresponding to those in Riemannian
submanifold geometry, that is, Gauss, Codazzi and Ricci equations. We have
also the following submersions:

semi-Riemannian submersion and Lorentzian submersion \cite{FAL}, Riemannian
submersion (\cite{GRAY}), slant submersion (\cite{CHEN}, \cite{SAHIN1}),
almost Hermitian submersion \cite{WAT}, contact-complex submersion \cite%
{IANUS}, quaternionic submersion \cite{IANUS2}, almost $h$-slant submersion
and $h$-slant submersion \cite{PARK1}, semi-invariant submersion \cite%
{SAHIN2}, $h$-semi-invariant submersion \cite{PARK2}, etc. As we know,
Riemannian submersions are related with physics and have their applications
in the Yang-Mills theory (\cite{BL}, \cite{WATSON}), Kaluza-Klein theory (%
\cite{BOL}, \cite{IV}), Supergravity and superstring theories (\cite{IV2}, 
\cite{MUS}). In \cite{SAHIN}, Sahin introduced anti-invariant Riemannian
submersions from almost Hermitian manifolds onto Riemannian manifolds.

In this paper we consider anti-invariant Riemannian submersions from
cosymplectic manifolds. The paper is organized as follows: In section 2, we
present the basic information about Riemannian submersions needed for this
paper. In section 3, we mention about cosymplectic manifolds. In section 4,
we give definition of anti-invariant Riemannian submersions and introduce
anti-invariant Riemannian submersions from cosymplectic manifolds onto
Riemannian manifolds. We survey main results of anti-invariant submersions
defined on cosymplectic manifolds. We give examples of anti-invariant
submersions such that characteristic vector field $\xi $ is vertical or
horizontal. Moreover we give decomposition theorems by using the existence
of anti-invariant Riemannian submersions and observe that such submersions
put some restrictions on the geometry of the total manifold.

\section{\textbf{Riemannian Submersions}}

In this section we recall several notions and results which will be needed
throughout the paper.

Let $(M,g_{M})$ be an $m$-dimensional Riemannian manifold , let $(N,g_{N})$
be an $n$-dimensional Riemannian manifold. A Riemannian submersion is a
smooth map $F:M\rightarrow N$ which is onto and satisfies the following
three axioms:

$S1$. $F$ has maximal rank.

$S2$. The differential $F_{\ast }$ preserves the lenghts of horizontal
vectors.

The fundamental tensors of a submersion were defined by O'Neill (\cite{BO1},%
\cite{BO2}). They are $(1,2)$-tensors on $M$, given by the formula:%
\begin{eqnarray}
\mathcal{T}(E,F) &=&\mathcal{T}_{E}F=\mathcal{H}\nabla _{\mathcal{V}E}%
\mathcal{V}F+\mathcal{V}\nabla _{\mathcal{V}E}\mathcal{H}F,  \label{AT1} \\
\mathcal{A}(E,F) &=&\mathcal{A}_{E}F=\mathcal{V}\nabla _{\mathcal{H}E}%
\mathcal{H}F+\mathcal{H}\nabla _{\mathcal{H}E}\mathcal{V}F,  \label{AT2}
\end{eqnarray}%
for any vector field $E$ and $F$ on $M.$ Here $\nabla $ denotes the
Levi-Civita connection of $(M,g_{M})$. These tensors are called
integrability tensors for the Riemannian submersions. Note that we denote
the projection morphism on the distributions ker$F_{\ast }$ and (ker$F_{\ast
})^{\perp }$ by $\mathcal{V}$ and $\mathcal{H},$ respectively.The following
\ Lemmas are well known (\cite{BO1},\cite{BO2}).

\begin{lemma}
For any $U,W$ vertical and $X,Y$ horizontal vector fields, the tensor fields 
$\mathcal{T}$, $\mathcal{A}$ satisfy:%
\begin{eqnarray}
i)\mathcal{T}_{U}W &=&\mathcal{T}_{W}U,  \label{TUW} \\
ii)\mathcal{A}_{X}Y &=&-\mathcal{A}_{Y}X=\frac{1}{2}\mathcal{V}\left[ X,Y%
\right] .  \label{TUW2}
\end{eqnarray}
\end{lemma}

It is easy to see that $\mathcal{T}$ $\ $is vertical, $\mathcal{T}_{E}=%
\mathcal{T}_{\mathcal{V}E}$ and $\mathcal{A}$ is horizontal, $\mathcal{A=A}_{%
\mathcal{H}E}$.

For each $q\in N,$ $F^{-1}(q)$ is an $(m-n)$ dimensional submanifold of $M$.
The submanifolds $F^{-1}(q),$ $q\in N,$ are called fibers. A vector field on 
$M$ is called vertical if it is always tangent to fibers. A vector field on $%
M$ is called horizontal if it is always orthogonal to fibers. A vector field 
$X$ on $M$ is called basic if $X$ is horizontal and $F$-related to a vector
field $X$ on $N,$ i. e., $F_{\ast }X_{p}=X_{\ast F(p)}$ for all $p\in M.$

\begin{lemma}
Let $F:(M,g_{M})\rightarrow (N,g_{N})$ be a Riemannian submersion. If $\ X,$ 
$Y$ are basic vector fields on $M$, then:
\end{lemma}

$i)$ $g_{M}(X,Y)=g_{N}(X_{\ast },Y_{\ast })\circ F,$

$ii)$ $\mathcal{H}[X,Y]$ is basic, $F$-related to $[X_{\ast },Y_{\ast }]$,

$iii)$ $\mathcal{H}(\nabla _{X}Y)$ is basic vector field corresponding to $%
\nabla _{X_{\ast }}^{^{\ast }}Y_{\ast }$ where $\nabla ^{\ast }$ is the
connection on $N.$

$iv)$ for any vertical vector field $V$, $[X,V]$ is vertical.

Moreover, if $X$ is basic and $U$ is vertical then $\mathcal{H}(\nabla
_{U}X)=\mathcal{H}(\nabla _{X}U)=\mathcal{A}_{X}U.$ On the other hand, from (%
\ref{AT1}) and (\ref{AT2}) we have%
\begin{eqnarray}
\nabla _{V}W &=&\mathcal{T}_{V}W+\hat{\nabla}_{V}W  \label{1} \\
\nabla _{V}X &=&\mathcal{H\nabla }_{V}X+\mathcal{T}_{V}X  \label{2} \\
\nabla _{X}V &=&\mathcal{A}_{X}V+\mathcal{V}\nabla _{X}V  \label{3} \\
\nabla _{X}Y &=&\mathcal{H\nabla }_{X}Y+\mathcal{A}_{X}Y  \label{4}
\end{eqnarray}%
for $X,Y\in \Gamma ((\ker F_{\ast })^{\bot })$ and $V,W\in \Gamma (\ker
F_{\ast }),$ where $\hat{\nabla}_{V}W=\mathcal{V}\nabla _{V}W.$

Notice that $\mathcal{T}$ \ acts on the fibres as the second fundamental
form of the submersion and restricted to vertical vector fields and it can
be easily seen that $\mathcal{T}=0$ is equivalent to the condition that the
fibres are totally geodesic. A Riemannian submersion is called a Riemannian
submersion with totally geodesic fiber if $\mathcal{T}$ $\ $vanishes
identically. Let $U_{1},...,U_{m-n}$ be an orthonormal frame of $\Gamma
(\ker F_{\ast }).$ Then the horizontal vector field $H$ $=\frac{1}{m-n}%
\sum\limits_{j=1}^{m-n}\mathcal{T}_{U_{j}}U_{j}$ is called the mean
curvature vector field of the fiber. If \ $H$ $=0$ the Riemannian submersion
is said to be minimal. A Riemannian submersion is called a Riemannian
submersion with totally umbilical fibers if 
\begin{equation}
\mathcal{T}_{U}W=g_{M}(U,W)H  \label{4a}
\end{equation}%
for $U,W\in $ $\Gamma (\ker F_{\ast })$. For any $E\in \Gamma (TM),\mathcal{T%
}_{E\text{ }}$and $\mathcal{A}_{E}$ are skew-symmetric operators on $(\Gamma
(TM),g_{M})$ reversing the horizontal and the vertical distributions. By
Lemma 1 horizontally distribution $\mathcal{H}$ is integrable if and only if
\ $\mathcal{A=}0$. For any $D,E,G\in \Gamma (TM)$ one has%
\begin{equation}
g(\mathcal{T}_{D}E,G)+g(\mathcal{T}_{D}G,E)=0,  \label{4b}
\end{equation}%
\begin{equation}
g(\mathcal{A}_{D}E,G)+g(\mathcal{A}_{D}G,E)=0.  \label{4c}
\end{equation}

We recall the notion of harmonic maps between Riemannian manifolds. Let $%
(M,g_{M})$ and $(N,g_{N})$ be Riemannian manifolds and suppose that $\varphi
:M\rightarrow N$ is a smooth map between them. Then the differential $%
\varphi _{\ast }$ of $\varphi $ can be viewed a section of the bundle $\
Hom(TM,\varphi ^{-1}TN)\rightarrow M,$ where $\varphi ^{-1}TN$ is the
pullback bundle which has fibres $(\varphi ^{-1}TN)_{p}=T_{\varphi (p)}N,$ $%
p\in M.\ Hom(TM,\varphi ^{-1}TN)$ has a connection $\nabla $ induced from
the Levi-Civita connection $\nabla ^{M}$ and the pullback connection. Then
the second fundamental form of $\varphi $ is given by 
\begin{equation}
(\nabla \varphi _{\ast })(X,Y)=\nabla _{X}^{\varphi }\varphi _{\ast
}(Y)-\varphi _{\ast }(\nabla _{X}^{M}Y)  \label{5}
\end{equation}%
for $X,Y\in \Gamma (TM),$ where $\nabla ^{\varphi }$ is the pullback
connection. It is known that the second fundamental form is symmetric. If $%
\varphi $ is a Riemannian submersion it can be easily prove that 
\begin{equation}
(\nabla \varphi _{\ast })(X,Y)=0  \label{5a}
\end{equation}%
for $X,Y\in \Gamma ((\ker F_{\ast })^{\bot })$. A smooth map $\varphi
:(M,g_{M})\rightarrow (N,g_{N})$ is said to be harmonic if $trace(\nabla
\varphi _{\ast })=0.$ On the other hand, the tension field of $\varphi $ is
the section $\tau (\varphi )$ of $\Gamma (\varphi ^{-1}TN)$ defined by%
\begin{equation}
\tau (\varphi )=div\varphi _{\ast }=\sum_{i=1}^{m}(\nabla \varphi _{\ast
})(e_{i},e_{i}),  \label{6}
\end{equation}%
where $\left\{ e_{1},...,e_{m}\right\} $ is the orthonormal frame on $M$.
Then it follows that $\varphi $ is harmonic if and only if $\tau (\varphi
)=0 $, for details, \cite{B}.

Let $g$ be a Riemannian metric tensor on the manifold $M=M_{1}\times M_{2}$
and assume that the canonical foliations \ $D_{M_{1}}$ and $D_{M_{2}}$
intersect perpendicularly everywhere. Then $g$ is the metric tensor of a
usual product of Riemannian manifolds if and only if $D_{M_{1}}$ and $%
D_{M_{2}}$ are totally geodesic foliations \cite{PON}.

\section{\textbf{Cosymplectic Manifolds}}

A $(2m+1)$-dimensional $C^{\infty }$-manifold $M$ is said to have an almost
contact structure if there exist on $M$ a tensor field $\phi $ of type $%
(1,1) $, a vector field $\xi $ and 1-form $\eta $ satisfying:%
\begin{equation}
\phi ^{2}=-I+\eta \otimes \xi ,\text{ \ }\phi \xi =0,\text{ }\eta \circ \phi
=0,\text{ \ \ }\eta (\xi )=1.  \label{fi}
\end{equation}%
There always exists a Riemannian metric $g$ on an almost contact manifold $M$
satisfying the following conditions%
\begin{equation}
g(\phi X,\phi Y)=g(X,Y)-\eta (X)\eta (Y),\text{\ }\eta (X)=g(X,\xi )
\label{metric}
\end{equation}%
where $X,Y$ are vector fields on $M.$

An almost contact structure $(\phi ,\xi ,\eta )$ is said to be normal if the
almost complex structure $J$ on the product manifold $M\times R$ is given by%
\begin{equation*}
J(X,f\frac{d}{dt})=(\phi X-f\xi ,\eta (X)\frac{d}{dt})
\end{equation*}%
where $f$ is the $C^{\infty }$-function on $M\times R$ has no torsion i.e., $%
J$ is integrable. The condition for normality in terms of $\phi ,\xi $ and $%
\eta $ is $\left[ \phi ,\phi \right] +2d\eta \otimes \xi =0$ on $M$, where $%
\left[ \phi ,\phi \right] $ is the Nijenhuis tensor of $\phi .$~Finally, the
fundamental two-form $\Phi $ is defined by $\Phi (X,Y)=g(X,\phi Y).$

An almost contact metric structure $(\phi ,\xi ,\eta ,g)$ is said to be
cosymplectic, if it is normal and both $\Phi $ and $\eta $ are closed (\cite%
{BL1}, \cite{LUDDEN}), and the structure equation of a cosymplectic manifold
is given by%
\begin{equation}
(\nabla _{X}\phi )Y=0  \label{nambla}
\end{equation}%
for any $X,Y$ tangent to $M,$ where $\nabla $ denotes the Riemannian
connection of the metric $g$ on $M.$ Moreover, for cosymplectic manifold%
\begin{equation}
\nabla _{X}\xi =0.  \label{xzeta}
\end{equation}

The canonical example of cosymplectic manifold is given by the product $%
B^{2n}\times 
%TCIMACRO{\U{211d} }%
%BeginExpansion
\mathbb{R}
%EndExpansion
$ Kahler manifold $B^{2n}(J,g)$ with $%
%TCIMACRO{\U{211d} }%
%BeginExpansion
\mathbb{R}
%EndExpansion
$ real line. Now we will introduce a well known cosymplectic manifold
example on $%
%TCIMACRO{\U{211d} }%
%BeginExpansion
\mathbb{R}
%EndExpansion
^{2n+1}.$

\begin{example}[\protect\cite{O}]
We consider $%
%TCIMACRO{\U{211d} }%
%BeginExpansion
\mathbb{R}
%EndExpansion
^{2n+1}$ with Cartesian coordinates $(x_{i},y_{i},z)$ $(i=1,...,n)$ and its
usual contact form 
\begin{equation*}
\eta =dz,
\end{equation*}%
The characteristic vector field $\xi $ is given by $\frac{\partial }{%
\partial z}$ and its Riemannian metric $g$ and tensor field $\phi $ are
given by%
\begin{equation*}
g=\sum\limits_{i=1}^{n}((dx_{i})^{2}+(dy_{i})^{2})+(dz)^{2},\text{ \ }\phi
=\left( 
\begin{array}{ccc}
0 & \delta _{ij} & 0 \\ 
-\delta _{ij} & 0 & 0 \\ 
0 & 0 & 0%
\end{array}%
\right) \text{, \ }i=1,...,n
\end{equation*}%
This gives a cosymplectic structure on $%
%TCIMACRO{\U{211d} }%
%BeginExpansion
\mathbb{R}
%EndExpansion
^{2n+1}$. The vector fields $E_{i}=\frac{\partial }{\partial y_{i}},$ $%
E_{n+i}=\frac{\partial }{\partial x_{i}}$, $\xi $ form a $\phi $-basis for
the cosymplectic structure. On the other hand, it can be shown that $%
%TCIMACRO{\U{211d} }%
%BeginExpansion
\mathbb{R}
%EndExpansion
^{2n+1}(\phi ,\xi ,\eta ,g)$ is a cosymplectic manifold.
\end{example}

\begin{example}[\protect\cite{KIM}]
We denote Cartesian coordinates in $%
%TCIMACRO{\U{211d} }%
%BeginExpansion
\mathbb{R}
%EndExpansion
^{5}$ by $(x_{1},x_{2},x_{3},x_{4},x_{5})$ and its Riemannian metric $g$ 
\begin{equation*}
g=\left( 
\begin{array}{ccccc}
1+\tau ^{2} & 0 & \tau ^{2} & 0 & -\tau \\ 
0 & 1 & 0 & 0 & 0 \\ 
\tau ^{2} & 0 & 1+\tau ^{2} & 0 & -\tau \\ 
0 & 0 & 0 & 1 & 0 \\ 
-\tau & 0 & -\tau & 0 & 1%
\end{array}%
\right) ,\text{ \ }
\end{equation*}%
where $\tau =\sin (x_{1}+x_{3}).$ We define \ an \ almost contact structure $%
(\phi ,\xi ,\eta )$ on $%
%TCIMACRO{\U{211d} }%
%BeginExpansion
\mathbb{R}
%EndExpansion
^{5}$ by 
\begin{equation*}
\phi =\left( 
\begin{array}{ccccc}
0 & -1 & 0 & 0 & 0 \\ 
1 & 0 & 0 & 0 & 0 \\ 
0 & 0 & 0 & -1 & 0 \\ 
0 & 0 & 1 & 0 & 0 \\ 
0 & -\tau & 0 & -\tau & 0%
\end{array}%
\right) ,\eta =-\tau dx_{1}-\tau dx_{3}+dx_{5},\xi =\frac{\partial }{%
\partial x_{5}}.
\end{equation*}%
The fundamental 2-form $\Phi $ have the form 
\begin{equation*}
\Phi =dx_{1}\wedge dx_{2}+dx_{3}\wedge dx_{4}.
\end{equation*}%
This gives a cosymplectic structure on $%
%TCIMACRO{\U{211d} }%
%BeginExpansion
\mathbb{R}
%EndExpansion
^{5}$. If we take vector fields $E_{1}=\frac{\partial }{\partial x_{1}}+\tau 
\frac{\partial }{\partial x_{5}},E_{2}=\frac{\partial }{\partial x_{3}},\phi
E_{1}=E_{3}=\frac{\partial }{\partial x_{2}},\phi E_{2}=E_{4}=\frac{\partial 
}{\partial x4}$ and $E_{5}=\frac{\partial }{\partial x_{5}}$ then these
vector fields form a frame field in $%
%TCIMACRO{\U{211d} }%
%BeginExpansion
\mathbb{R}
%EndExpansion
^{5}$.
\end{example}

\section{\textbf{Anti-invariant Riemannian submersions}}

\begin{definition}
Let $M(\phi ,\xi ,\eta ,g_{M})$ be a cosymplectic manifold and $(N,g_{N})$
be a Riemannian manifold. A Riemannian submersion $F:M(\phi ,\xi ,\eta
,g_{M})\rightarrow $ $(N,g_{N})$ is called an anti-invariant Riemannian
submersion if $\ker F_{\ast }$ is anti-invariant with respect to $\phi $,
i.e. $\phi (\ker F_{\ast })\subseteq (\ker F_{\ast })^{\bot }.$
\end{definition}

Let $F:M(\phi ,\xi ,\eta ,g_{M})\rightarrow $ $(N,g_{N})$ be an
anti-invariant Riemannian submersion from a cosymplectic manifold $M(\phi
,\xi ,\eta ,g_{M})$ to a Riemannian manifold $(N,g_{N}).$ First of all, from
Definition 1, we have $\phi (\ker F_{\ast })\cap (\ker F_{\ast })^{\bot
}\neq \left\{ 0\right\} .$ We denote the complementary orthogonal
distribution to $\phi (\ker F_{\ast })$ in $(\ker F_{\ast })^{\bot }$ by $%
\mu .$ Then we have%
\begin{equation}
(\ker F_{\ast })^{\bot }=\phi \ker F_{\ast }\oplus \mu .  \label{A1}
\end{equation}%
Now we will introduce some examples.\ \ \ \ \ \ \ \ 

\begin{example}
$%
%TCIMACRO{\U{211d} }%
%BeginExpansion
\mathbb{R}
%EndExpansion
^{5}$ \ has got a cosymplectic structure as in Example 1.

Let $F:%
%TCIMACRO{\U{211d} }%
%BeginExpansion
\mathbb{R}
%EndExpansion
^{5}\rightarrow 
%TCIMACRO{\U{211d} }%
%BeginExpansion
\mathbb{R}
%EndExpansion
^{2}$ be a map defined by $F(x_{1},x_{2},y_{1},y_{2},z)=(\frac{x_{1}+y_{2}}{%
\sqrt{2}},\frac{x_{2}+y_{1}}{\sqrt{2}})$. Then, by direct calculations 
\begin{equation*}
\ker F_{\ast }=span\{V_{1}=\frac{1}{\sqrt{2}}(E_{1}-E_{4}),\text{ }V_{2}=%
\frac{1}{\sqrt{2}}(E_{2}-E_{3}),\text{ }V_{3}=E_{5}=\xi \}
\end{equation*}%
and 
\begin{equation*}
(\ker F_{\ast })^{\bot }=span\{H_{1}=\frac{1}{\sqrt{2}}(E_{1}+E_{4}),\text{ }%
H_{2}=\frac{1}{\sqrt{2}}(E_{2}+E_{3})\}.
\end{equation*}%
Then it is easy to see that $F$ is a Riemannian submersion. Moreover, $\phi
V_{1}=H_{2},$ $\phi V_{2}=H_{1},$ $\phi V_{3}=0$ imply that $\phi (\ker
F_{\ast })=(\ker F_{\ast })^{\bot }.$ As a result, $F$ is an anti-invariant
Riemannian submersion such that $\xi $ is vertical.
\end{example}

\begin{example}
$%
%TCIMACRO{\U{211d} }%
%BeginExpansion
\mathbb{R}
%EndExpansion
^{5}$ be a cosymplectic manifold as in Example 2.

Let $F:%
%TCIMACRO{\U{211d} }%
%BeginExpansion
\mathbb{R}
%EndExpansion
^{5}\rightarrow 
%TCIMACRO{\U{211d} }%
%BeginExpansion
\mathbb{R}
%EndExpansion
^{2}$ be a map defined by $%
F(x_{1},x_{2},x_{3},x_{4},x_{5})=(x_{1}+x_{2},x_{3}+x_{4})$. \ After some
calculations we have%
\begin{equation*}
\ker F_{\ast }=span\{V_{1}=E_{1}-E_{3},\text{ }V_{2}=E_{2}-E_{4},V_{3}=\xi \}
\end{equation*}%
and%
\begin{equation*}
(\ker F_{\ast })^{\bot }=span\{H_{1}=E_{1}+E_{3},\text{ }H_{2}=E_{2}+E_{4}\}
\end{equation*}%
Then it is easy to see that $F$ is a Riemannian submersion. Moreover, $\phi
V_{1}=H_{1},$ $\phi V_{2}=H_{2},$ $\phi V_{3}=0$ imply that $\phi (\ker
F_{\ast })=(\ker F_{\ast })^{\bot }.$ As a result, $F$ is an anti-invariant
Riemannian submersion such that $\xi $ is vertical.
\end{example}

\begin{example}
$%
%TCIMACRO{\U{211d} }%
%BeginExpansion
\mathbb{R}
%EndExpansion
^{7}$ be a cosymplectic manifold as in Example 1.

Let $F:%
%TCIMACRO{\U{211d} }%
%BeginExpansion
\mathbb{R}
%EndExpansion
^{7}\rightarrow 
%TCIMACRO{\U{211d} }%
%BeginExpansion
\mathbb{R}
%EndExpansion
^{4}$ be a map defined by $F(x_{1},x_{2},x_{3},y_{1},y_{2},y_{3},z)=(\frac{1%
}{\sqrt{2}}(x_{1}+y_{1}),\frac{1}{\sqrt{2}}(x_{2}+y_{2}),\frac{1}{\sqrt{2}}%
(x_{3}+y_{3}),\frac{1}{\sqrt{2}}(x_{3}-y_{3}))$. \ After some calculations
we have%
\begin{equation*}
\ker F_{\ast }=span\{V_{1}=\frac{1}{\sqrt{2}}(E_{1}-E_{4}),V_{2}=\frac{1}{%
\sqrt{2}}(E_{2}-E_{5}),V_{3}=\xi \}
\end{equation*}%
and%
\begin{equation*}
(\ker F_{\ast })^{\bot }=span\{H_{1}=\frac{1}{\sqrt{2}}(E_{1}+E_{4}),\text{ }%
H_{2}=\frac{1}{\sqrt{2}}(E_{2}+E_{5}),\text{ }H_{3}=\frac{1}{\sqrt{2}}%
(E_{3}-E_{6}),H_{4}=\frac{1}{\sqrt{2}}(E_{3}+E_{6})\}.
\end{equation*}%
Then it is easy to see that $F$ is a Riemannian submersion. Moreover, $\phi
V_{1}=H_{1}$, $\phi V_{2}=H_{2}$ imply that $\phi (\ker F_{\ast })\subset
(\ker F_{\ast })^{\bot }=\phi (\ker F_{\ast })\oplus span\{H_{3},H_{4}\}.$%
Hence $F$ is an anti-invariant Riemannian submersion such that $\xi $ is
vertical.
\end{example}

\begin{example}
$%
%TCIMACRO{\U{211d} }%
%BeginExpansion
\mathbb{R}
%EndExpansion
^{5}$ be a cosymplectic manifold as in Example 1.

Let $F:%
%TCIMACRO{\U{211d} }%
%BeginExpansion
\mathbb{R}
%EndExpansion
^{5}\rightarrow 
%TCIMACRO{\U{211d} }%
%BeginExpansion
\mathbb{R}
%EndExpansion
^{3}$ be a map defined by $F(x_{1},x_{2},y_{1},y_{2},z)=(\frac{x_{1}+y_{2}}{%
\sqrt{2}},\frac{x_{2}+y_{1}}{\sqrt{2}},z)$. \ After some calculations we have%
\begin{equation*}
\ker F_{\ast }=span\{V_{1}=\frac{1}{\sqrt{2}}(E_{1}-E_{4}),\text{ }V_{2}=%
\frac{1}{\sqrt{2}}(E_{2}-E_{3})\}
\end{equation*}%
and%
\begin{equation*}
(\ker F_{\ast })^{\bot }=span\{H_{1}=\frac{1}{\sqrt{2}}(E_{1}+E_{4}),\text{ }%
H_{2}=\frac{1}{\sqrt{2}}(E_{2}+E_{3}),\text{ }H_{3}=E_{5}=\xi \}
\end{equation*}%
Then it is easy to see that $F$ is a Riemannian submersion. Moreover, $\phi
V_{1}=H_{2}$, $\phi V_{2}=H_{1}$ imply that $\phi (\ker F_{\ast })\subset
(\ker F_{\ast })^{\bot }=$ $\phi (\ker F_{\ast })\oplus \{\xi \}$. Thus $F$
is an anti-invariant Riemannian submersion such that $\xi $ is horizontal.
\end{example}

\subsection{\textbf{Anti-invariant submersions admitting vertical structure
vector field }}

In this section, we will study anti-invariant submersions from a
cosymplectic manifold onto a Riemannian manifold such that the
characteristic vector field $\xi $ is vertical.

It is easy to see that $\mu $ is an invariant distribution of $(\ker F_{\ast
})^{\bot },$ under the endomorphism $\phi $. Thus, for $X\in \Gamma ((\ker
F_{\ast })^{\bot }),$ we write%
\begin{equation}
\phi X=BX+CX,  \label{A2}
\end{equation}%
where $BX\in \Gamma (\ker F_{\ast })$ and $CX\in \Gamma (\mu ).$ On the
other hand, since $F_{\ast }((\ker F_{\ast })^{\bot })=TN$ and $F$ is a
Riemannian submersion, using (\ref{A2}) we derive $g_{N}(F_{\ast }\phi
V,F_{\ast }CX)=0,$ for every $X\in $ $\Gamma ((\ker F_{\ast }))^{\perp \text{
}}$and $V\in \Gamma (\ker F_{\ast })$, which implies that 
\begin{equation}
TN=F_{\ast }(\phi (\ker F_{\ast }))\oplus F_{\ast }(\mu ).  \label{A2a}
\end{equation}

\begin{theorem}
Let $M(\phi ,\xi ,\eta ,g_{M})$ be a cosymplectic manifold \ of dimension $%
2m+1$ and $(N,g_{N})$ is a Riemannian manifold of dimension $n$. Let $%
F:M(\phi ,\xi ,\eta ,g_{M})\rightarrow $ $(N,g_{N})$ be an anti-invariant
Riemannian submersion such that $\phi (\ker F_{\ast })=(\ker F_{\ast
})^{\bot }$. Then the characteristic vector field $\xi $ is vertical and $%
m=n $.
\end{theorem}

\begin{proof}
By the assumption $\phi (\ker F_{\ast })=(\ker F_{\ast })^{\bot }$, for any $%
U\in \Gamma (\ker F_{\ast })$ we have $g_{M}(\xi ,\phi U)=-g_{M}(\phi \xi
,U)=0$, which shows that the structure vector field is vertical. Now we
suppose that $U_{1},...,U_{k-1},\xi =U_{k}$ be an orthonormal frame of $%
\Gamma (\ker F_{\ast })$, where $k=2m-n+1$. Since $\phi (\ker F_{\ast
})=(\ker F_{\ast })^{\bot }$, $\phi U_{1},...,\phi U_{k-1}$ form an
orthonormal frame of $\Gamma ((\ker F_{\ast })^{\bot })$. So, by help of (%
\ref{A2a}) $\ $we obtain $k=n+1$ which implies that $m=n$.
\end{proof}

\begin{remark}
We note that Example 3 and Example 4 satisfy Theorem 1.
\end{remark}

From (\ref{fi}) and (\ref{A2}) we have following Lemma.

\begin{lemma}
Let $F$ be an anti-invariant Riemannian submersion from a cosymplectic
manifold $M(\phi ,\xi ,\eta ,g_{M})$ to a Riemannian manifold $(N,g_{N})$.
Then we have%
\begin{eqnarray*}
BCX &=&0,\text{ }\eta (BX)=0, \\
C^{2}X &=&-X-\phi BX, \\
C\phi V &=&0,\text{ }C^{3}X+CX=0, \\
B\phi V &=&-V+\eta (V)\xi
\end{eqnarray*}%
for any $X\in \Gamma ((\ker F_{\ast })^{\bot })$ and $V\in \Gamma ((\ker
F_{\ast })).$
\end{lemma}

Using (\ref{nambla}) one can easily obtain 
\begin{equation}
\nabla _{X}Y=-\phi \nabla _{X}\phi Y  \label{Namblafi2}
\end{equation}%
for $X,Y\in \Gamma ((\ker F_{\ast })^{\bot }).$

\begin{lemma}
Let $F$ be an anti-invariant Riemannian submersion from a cosymplectic
manifold $M(\phi ,\xi ,\eta ,g_{M})$ to a Riemannian manifold $(N,g_{N})$.
Then we have%
\begin{equation}
\mathcal{A}_{X}\xi =0,  \label{C11}
\end{equation}%
\begin{equation}
\mathcal{T}_{U}\xi =0,  \label{C1,5}
\end{equation}%
\begin{equation}
g_{M}(CX,\phi U)=0,  \label{C2}
\end{equation}%
\begin{equation}
g_{M}(\nabla _{X}CY,\phi U)=-g_{M}(CY,\phi \mathcal{A}_{X}U)  \label{C3}
\end{equation}%
for any $X,Y\in \Gamma ((\ker F_{\ast })^{\bot })$ and $U\in \Gamma ((\ker
F_{\ast })).$
\end{lemma}

\begin{proof}
By virtue of (\ref{1}) and (\ref{xzeta}) we have (\ref{C11}). Using (\ref{3}%
) and (\ref{xzeta}) we get (\ref{C1,5}).

For $X\in \Gamma ((\ker F_{\ast })^{\bot })$ and $U\in \Gamma (\ker F_{\ast
})$, by virtue of (\ref{metric}) and (\ref{A2}) we get%
\begin{eqnarray}
g_{M}(CX,\phi U) &=&g_{M}(\phi X-BX,\phi U)  \label{C5} \\
&=&g_{M}(X,U)-\eta (X)\eta (U)+g_{M}(\phi BX,U).  \notag
\end{eqnarray}%
Since $\phi BX\in \Gamma ((\ker F_{\ast })^{\bot })$ and $\xi \in \Gamma
(\ker F_{\ast }),$ (\ref{C5}) implies (\ref{C2}).

Then using (\ref{3}), (\ref{nambla}) and (\ref{C2}), we have 
\begin{equation*}
g_{M}(\nabla _{X}CY,\phi U)=-g_{M}(CY,\phi \mathcal{A}_{X}U)-g_{M}(CY,\phi (%
\mathcal{V}\nabla _{X}U)).
\end{equation*}%
Since $\phi (\mathcal{V}\nabla _{X}U)\in \Gamma (\phi \ker F_{\ast })=\Gamma
((\ker F_{\ast })^{\bot }),$ we obtain (\ref{C3}).
\end{proof}

\begin{theorem}
Let $M(\phi ,\xi ,\eta ,g_{M})$ be a cosymplectic manifold\ of dimension $%
2m+1$ and $(N,g_{N})$ is a Riemannian manifold of dimension $n$. Let $%
F:M(\phi ,\xi ,\eta ,g_{M})\rightarrow $ $(N,g_{N})$ be an anti-invariant
Riemannian submersion. Then the fibers are not proper totally umbilical.
\end{theorem}

\begin{proof}
If the fibers are proper totally umbilical, then we have $\mathcal{T}%
_{U}V=g_{M}(U,V)H$ for any vertical vector fields $U,V$ where $H$ is the
mean curvature vector field of any fibre. Since $\mathcal{T}_{\xi }\xi $ $=0$%
, we have $H=0$, which shows that fibres are minimal. Hence the fibers are
totally geodesic. This completes proof of the Theorem.
\end{proof}

Since the distribution $\ker F_{\ast }$ is integrable, we only study the
integrability of the distribution $(\ker F_{\ast })^{\bot }$ and then we
investigate the geometry of leaves of $\ker F_{\ast }$ and $(\ker F_{\ast
})^{\bot }.$

\begin{theorem}
Let $F$ be an anti-invariant Riemannian submersion from a cosymplectic
manifold $M(\phi ,\xi ,\eta ,g_{M})$ to a Riemannian manifold $(N,g_{N})$.
Then the following assertions are equivalent to each other;
\end{theorem}

$\ i)$ $(\ker F_{\ast })^{\bot }$ \textit{is integrable. }

$\ $

$ii)$ 
\begin{eqnarray*}
g_{N}((\nabla F_{\ast })(Y,BX),F_{\ast }\phi V) &=&g_{N}((\nabla F_{\ast
})(X,BY),F_{\ast }\phi V) \\
&&+g_{M}(CY,\phi \mathcal{A}_{X}V)-g_{M}(CX,\phi \mathcal{A}_{Y}V).
\end{eqnarray*}

$iii)$%
\begin{equation*}
g_{M}(\mathcal{A}_{X}BY-\mathcal{A}_{Y}BX,\phi V)=g_{M}(CY,\phi \mathcal{A}%
_{X}V)-g_{M}(CX,\phi \mathcal{A}_{Y}V).
\end{equation*}%
\textit{for }$X,Y\in \Gamma ((\ker F_{\ast })^{\bot })$\textit{\ and }$V\in
\Gamma (\ker F_{\ast }).$

\begin{proof}
Using (\ref{Namblafi2}), for $X,Y\in \Gamma ((\ker F_{\ast })^{\bot })$ and $%
V\in \Gamma (\ker F_{\ast })$ we get%
\begin{eqnarray*}
g_{M}(\left[ X,Y\right] ,V) &=&g_{M}(\nabla _{X}Y,V)-g_{M}(\nabla _{Y}X,V) \\
&=&g_{M}(\nabla _{X}\phi Y,\phi V)-g_{M}(\nabla _{Y}\phi X,\phi V).
\end{eqnarray*}%
Then from (\ref{A2}) we have%
\begin{eqnarray*}
g_{M}(\left[ X,Y\right] ,V) &=&g_{M}(\nabla _{X}BY,\phi V)+g_{M}(\nabla
_{X}CY,\phi V)-g_{M}(\nabla _{Y}BX,\phi V) \\
&&-g_{M}(\nabla _{Y}CX,\phi V).
\end{eqnarray*}%
Using (\ref{AT2}), (\ref{3}) and (\ref{C3}) and if we take into account that 
$F$ is a Riemannian submersion, we obtain%
\begin{eqnarray*}
g_{M}(\left[ X,Y\right] ,V) &=&g_{N}(F_{\ast }\nabla _{X}BY,F_{\ast }\phi
V)-g_{M}(CY,\phi \mathcal{A}_{X}V) \\
&&-g_{N}(F_{\ast }\nabla _{Y}BX,F_{\ast }\phi V)+g_{M}(CX,\phi \mathcal{A}%
_{Y}V).
\end{eqnarray*}%
Thus, from (\ref{5}) we have 
\begin{eqnarray*}
g_{M}(\left[ X,Y\right] ,V) &=&g_{N}(-(\nabla F_{\ast })(X,BY)+(\nabla
F_{\ast })(Y,BX),F_{\ast }\phi V) \\
&&+g_{M}(CX,\phi \mathcal{A}_{Y}V)-g_{M}(CY,\phi \mathcal{A}_{X}V)
\end{eqnarray*}%
which proves $(i)\Leftrightarrow (ii).$ On the other hand using (\ref{5}) we
get%
\begin{equation*}
(\nabla F_{\ast })(Y,BX)-(\nabla F_{\ast })(X,BY)=-F_{\ast }(\nabla
_{Y}BX-\nabla _{X}BY).
\end{equation*}%
Then (\ref{3}) implies that 
\begin{equation*}
(\nabla F_{\ast })(Y,BX)-(\nabla F_{\ast })(X,BY)=-F_{\ast }(\mathcal{A}%
_{Y}BX-\mathcal{A}_{X}BY).
\end{equation*}

From (\ref{AT2}) $\mathcal{A}_{Y}BX-\mathcal{A}_{X}BY\in \Gamma ((\ker
F_{\ast })^{\bot }),$ this shows that $(ii)\Leftrightarrow (iii).$
\end{proof}

\begin{remark}
If $\phi (\ker F_{\ast })=(\ker F_{\ast })^{\bot }$ then we get $C=0$ and
morever (\ref{A2a}) implies that $TN=F_{\ast }(\phi (\ker F_{\ast }))$.
\end{remark}

Hence we have the following Corollary.

\begin{corollary}
Let $F:M(\phi ,\xi ,\eta ,g_{M})\rightarrow $ $(N,g_{N})$ be an
anti-invariant Riemannian submersion such that $\phi (\ker F_{\ast })=(\ker
F_{\ast })^{\bot },$ where $M(\phi ,\xi ,\eta ,g_{M})$ is a cosymplectic
manifold and $(N,g_{N})$ is a Riemannian manifold. Then following assertions
are equivalent to each other;

$i)(\ker F_{\ast })^{\bot }$ is\ integrable\textit{.}

$ii)(\nabla F_{\ast })(Y,\phi X)=(\nabla F_{\ast })(X,\phi Y)$ for $X,Y\in
\Gamma ((\ker F_{\ast })^{\bot }).$

$iii)\mathcal{A}_{X}\phi Y=\mathcal{A}_{Y}\phi X.$
\end{corollary}

\begin{theorem}
Let $F$ be an anti-invariant Riemannian submersion from a cosymplectic
manifold $M(\phi ,\xi ,\eta ,g_{M})$ to a Riemannian manifold $(N,g_{N}).$
Then the following assertions are equivalent to each other;
\end{theorem}

$\ i)$ $(\ker F_{\ast })^{\bot }$ \textit{defines a totally geodesic
foliation on }$M.$

$\ ii)$ 
\begin{equation*}
g_{M}(\mathcal{A}_{X}BY,\phi V)=g_{M}(CY,\phi \mathcal{A}_{X}V).
\end{equation*}

$iii)$%
\begin{equation*}
g_{N}((\nabla F_{\ast })(X,\phi Y),F_{\ast }\phi V)=-g_{M}(CY,\phi \mathcal{A%
}_{X}V).
\end{equation*}%
\textit{for }$X,Y\in \Gamma ((\ker F_{\ast })^{\bot })$\textit{\ and }$V\in
\Gamma (\ker F_{\ast }).$

\begin{proof}
From (\ref{metric}) and (\ref{nambla}) we obtain%
\begin{equation*}
g_{M}(\nabla _{X}Y,V)=g_{M}(\nabla _{X}\phi Y,\phi V)
\end{equation*}%
for\textit{\ }$X,Y\in \Gamma ((\ker F_{\ast })^{\bot })$\textit{\ }and%
\textit{\ }$V\in \Gamma (\ker F_{\ast }).$Using (\ref{3}) and (\ref{A2})%
\begin{equation*}
g_{M}(\nabla _{X}Y,V)=g_{M}(\mathcal{A}_{X}BY+\mathcal{V}\nabla _{X}BY,\phi
V)-g_{M}(CY,\phi \mathcal{A}_{X}V).
\end{equation*}

The last equation shows $(i)\Leftrightarrow (ii)$.

For $X,Y\in \Gamma ((\ker F_{\ast })^{\bot })$\textit{\ }and\textit{\ }$V\in
\Gamma (\ker F_{\ast }),$%
\begin{eqnarray}
g_{M}(\mathcal{A}_{X}BY,\phi V) &=&g_{M}(CY,\phi \mathcal{A}_{X}V)  \notag \\
&&\overset{(\ref{C3})}{=}-g_{M}(\nabla _{X}CY,\phi V)  \notag \\
&&\overset{(\ref{A2})}{=}-g_{M}(\nabla _{X}\phi Y,\phi V)+g_{M}(\nabla
_{X}BY,\phi V)  \label{B6}
\end{eqnarray}%
Since differential $F_{\ast }$ preserves the lenghts of horizontal vectors
the relation (\ref{B6}) forms%
\begin{equation}
g_{M}(\mathcal{A}_{X}BY,\phi V)=g_{N}(F_{\ast }\nabla _{X}BY,F_{\ast }\phi
V)-g_{M}(\nabla _{X}\phi Y,\phi V)  \label{B7}
\end{equation}%
Using, (\ref{nambla}), (\ref{metric}), (\ref{5}) and (\ref{5a}) in (\ref{B7}%
) respectively, we obtain%
\begin{equation*}
g_{M}(\mathcal{A}_{X}BY,\phi V)=g_{N}(-(\nabla F_{\ast })(X,\phi Y),F_{\ast
}\phi V)
\end{equation*}%
which tells that $(ii)\Leftrightarrow (iii).$
\end{proof}

\begin{corollary}
Let $F:M(\phi ,\xi ,\eta ,g_{M})\rightarrow $ $(N,g_{N})$ be an
anti-invariant Riemannian submersion such that $\phi (\ker F_{\ast })=(\ker
F_{\ast })^{\bot },$ where $M(\phi ,\xi ,\eta ,g_{M})$ is a cosymplectic
manifold and $(N,g_{N})$ is a Riemannian manifold. Then the following
assertions are equivalent to each other;

$\ i)$ $(\ker F_{\ast })^{\bot }$ \textit{defines a totally geodesic
foliation on }$M.$

$ii)$ $\mathcal{A}_{X}\phi Y=0.$

$iii)$ $(\nabla F_{\ast })(X,\phi Y)=0$ \textit{for }$X,Y\in \Gamma ((\ker
F_{\ast })^{\bot })$\textit{\ and }$V\in \Gamma (\ker F_{\ast })$.
\end{corollary}

\begin{theorem}
Let $F$ be an anti-invariant Riemannian submersion from a cosymplectic
manifold $M(\phi ,\xi ,\eta ,g_{M})$ to a Riemannian manifold $(N,g_{N}).$
Then the following assertions are equivalent to each other;

$i)$ $(\ker F_{\ast })$ defines a totally geodesic foliation on $M$ .

$ii)$ $g_{N}((\nabla F_{\ast })(V,\phi X),F_{\ast }\phi W)=0$ \ \ for $X\in
\Gamma ((\ker F_{\ast })^{\bot })$ and $V,W\in \Gamma (\ker F_{\ast }).$

$iii)$ $\mathcal{T}_{V}BX+\mathcal{A}_{CX}V\in \Gamma (\mu ).$
\end{theorem}

\begin{proof}
Since $g_{M}(W,X)=0$ we have $g_{M}(\nabla _{V}W,X)=-g(W,\nabla _{V}X).$
From (\ref{metric}) and (\ref{nambla}) we get $g_{M}(\nabla
_{V}W,X)=-g_{M}(\phi W,H\nabla _{V}\phi X).$ Then Riemannian submersion $F$
and (\ref{5}) imply that 
\begin{equation*}
g_{M}(\nabla _{V}W,X)=g_{N}(F_{\ast }\phi W,(\nabla F_{\ast })(V,\phi X))
\end{equation*}%
which is $(i)\Leftrightarrow (ii).$ By direct calculation, we derive%
\begin{equation*}
g_{N}((F_{\ast }\phi W,(\nabla F_{\ast })(V,\phi X))=-g_{M}(\phi W,\nabla
_{V}\phi X).
\end{equation*}%
Using (\ref{A2}) we have 
\begin{equation*}
g_{N}((F_{\ast }\phi W,(\nabla F_{\ast })(V,\phi X))=-g_{M}(\phi W,\nabla
_{V}BX+\nabla _{V}CX).
\end{equation*}%
Hence we get%
\begin{equation*}
g_{N}((F_{\ast }\phi W,(\nabla F_{\ast })(V,\phi X))=-g_{M}(\phi W,\nabla
_{V}BX+\left[ V,CX\right] +\nabla _{CX}V).
\end{equation*}%
Since $\left[ V,CX\right] \in \Gamma (\ker F_{\ast }),$ using (\ref{1}) and (%
\ref{3}), we obtain%
\begin{equation*}
g_{N}((F_{\ast }\phi W,(\nabla F_{\ast })(V,\phi X))=-g_{M}(\phi W,\mathcal{T%
}_{V}BX+\mathcal{A}_{CX}V).
\end{equation*}%
This shows $(ii)\Leftrightarrow (iii).$
\end{proof}

\begin{corollary}
Let $F:M(\phi ,\xi ,\eta ,g_{M})\rightarrow $ $(N,g_{N})$ be an
anti-invariant Riemannian submersion such that $\phi (\ker F_{\ast })=(\ker
F_{\ast })^{\bot },$ where $M(\phi ,\xi ,\eta ,g_{M})$ is a cosymplectic
manifold and $(N,g_{N})$ is a Riemannian manifold. Then following assertions
are equivalent to each other;

$i)(\ker F_{\ast })$ \textit{defines a totally geodesic foliation on }$M.$

$ii)(\nabla F_{\ast })(V,\phi X)=0,$ for $X\in \Gamma ((\ker F_{\ast
})^{\bot })$ and $V,W\in \Gamma (\ker F_{\ast }).$

$iii)$ $\mathcal{T}_{V}\phi W=0.$
\end{corollary}

We note that a differentiable map $F$ between two Riemannian manifolds is
called totally geodesic if $\ \nabla F_{\ast }=0.$ For\ an anti-invariant
Riemannian submersion such that $\phi (\ker F_{\ast })=(\ker F_{\ast
})^{\bot }$ we have the following characterization.

\begin{theorem}
Let $F:M(\phi ,\xi ,\eta ,g_{M})\rightarrow $ $(N,g_{N})$ be an
anti-invariant Riemannian submersion such that $\phi (\ker F_{\ast })=(\ker
F_{\ast })^{\bot },$ where $M(\phi ,\xi ,\eta ,g_{M})$ is a cosymplectic
manifold and $(N,g_{N})$ is a Riemannian manifold. Then $F$ is a totally
geodesic map if and only if%
\begin{equation}
\mathcal{T}_{W}\phi V=0,\text{ \ \ }\forall W,\text{ }V\in \Gamma (\ker
F_{\ast })  \label{BE1}
\end{equation}%
and%
\begin{equation}
\mathcal{A}_{X}\phi W=0,\text{ \ }\forall X\in \Gamma ((\ker F_{\ast
})^{\bot }),\forall W\in \Gamma (\ker F_{\ast }).\text{\ }  \label{BE2}
\end{equation}
\end{theorem}

\begin{proof}
First of all, we recall that the second fundamental form of a Riemannian
submersion satisfies (\ref{5a}). For $W,$ $V\in \Gamma (\ker F_{\ast }),$ by
using (\ref{2}), (\ref{5}), (\ref{fi}) and (\ref{nambla}), we get%
\begin{equation}
(\nabla F_{\ast })(W,V)=F_{\ast }(\phi \mathcal{T}_{W}\phi V).  \label{BE3}
\end{equation}%
On the other hand by using (\ref{5}) and (\ref{nambla}) we have%
\begin{equation*}
(\nabla F_{\ast })(X,W)=F_{\ast }(\phi \nabla _{X}\phi W)
\end{equation*}%
for $X\in \Gamma ((\ker F_{\ast })^{\bot }).$ Then from (\ref{4}) and (\ref%
{fi}), we obtain 
\begin{equation}
(\nabla F_{\ast })(X,W)=F_{\ast }(\phi \mathcal{A}_{X}\phi W).  \label{BE4}
\end{equation}%
Since $\phi $ is non-singular, using (\ref{4b}) and (\ref{4c}) proof comes
from (\ref{5a}), (\ref{BE3}) and (\ref{BE4}).
\end{proof}

Finally, we give a necessary and sufficient condition for an anti-invariant
Riemannian submersion such that $\phi (\ker F_{\ast })=(\ker F_{\ast
})^{\bot }$ to be harmonic.

\begin{theorem}
Let $F:M(\phi ,\xi ,\eta ,g_{M})\rightarrow $ $(N,g_{N})$ be an
anti-invariant Riemannian submersion such that $\phi (\ker F_{\ast })=(\ker
F_{\ast })^{\bot },$ where $M(\phi ,\xi ,\eta ,g_{M})$ is a cosymplectic
manifold and $(N,g_{N})$ is a Riemannian manifold. Then $F$ is harmonic if
and only if Trace$\phi \mathcal{T}_{V}=0$ for $V\in \Gamma (\ker F_{\ast }).$
\end{theorem}

\begin{proof}
From \cite{EJ} we know that $F$ is harmonic if and only if $F$ has minimal
fibres. Thus $F$ is harmonic if and only if $\sum\limits_{i=1}^{k}\mathcal{T}%
_{e_{i}}e_{i}=0,$ where $k$ is dimension of $\ker F_{\ast }$. On the other
hand, from (\ref{1}), (\ref{2}) and (\ref{nambla}), we get%
\begin{equation}
\mathcal{T}_{V}\phi W=\phi \mathcal{T}_{V}W  \label{S15}
\end{equation}%
for any $W,$ $V\in \Gamma (\ker F_{\ast }).$ Using (\ref{S15}), we get 
\begin{equation*}
\sum\limits_{i=1}^{k}g_{M}(\mathcal{T}_{e_{i}}\phi
e_{i},V)=-\sum\limits_{i=1}^{k}g_{M}(\mathcal{T}_{e_{i}}e_{i},\phi V)
\end{equation*}%
for any $V\in \Gamma (\ker F_{\ast }).$ (\ref{4b}) implies that%
\begin{equation*}
\sum\limits_{i=1}^{k}g_{M}(\phi e_{i},\mathcal{T}_{e_{i}}V)=\sum%
\limits_{i=1}^{k}g_{M}(\mathcal{T}_{e_{i}}e_{i},\phi V).
\end{equation*}%
Then, using (\ref{TUW}) we have%
\begin{equation*}
\sum\limits_{i=1}^{k}g_{M}(\phi e_{i},\mathcal{T}_{Ve_{i}})=\sum%
\limits_{i=1}^{k}g_{M}(\mathcal{T}_{e_{i}}e_{i},\phi V).
\end{equation*}%
Hence, proof comes from (\ref{metric}).
\end{proof}

Using \cite{PON}, Theorem 4 and Theorem 5 we will give our first
decomposition theorem for an anti invariant Riemannian submersion.

\begin{theorem}
Let $F$ be an anti-invariant Riemannian submersion from a cosymplectic
manifold $M(\phi ,\xi ,\eta ,g_{M})$ to a Riemannian manifold $(N,g_{N}).$
Then $M$ is a locally product manifold if and only if%
\begin{equation*}
g_{N}((\nabla F_{\ast })(X,\phi Y),F_{\ast }\phi V)=-g_{M}(CY,\phi \mathcal{A%
}_{X}V)
\end{equation*}%
and%
\begin{equation*}
g_{N}((\nabla F_{\ast })(V,\phi X),F_{\ast }\phi W)=0
\end{equation*}%
for $W,$ $V\in \Gamma (\ker F_{\ast }),$ $X,$ $Y\in \Gamma ((\ker F_{\ast
})^{\bot }).$
\end{theorem}

From Corollary 2 and Corollary 3 we obtain following decomposition theorem.

\begin{theorem}
Let $F:M(\phi ,\xi ,\eta ,g_{M})\rightarrow $ $(N,g_{N})$ be an
anti-invariant Riemannian submersion such that $\phi (\ker F_{\ast })=(\ker
F_{\ast })^{\bot },$ where $M(\phi ,\xi ,\eta ,g_{M})$ is a cosymplectic
manifold and $(N,g_{N})$ is a Riemannian manifold. Then $M$ is a locally
product manifold if and only if \ $\mathcal{A}_{X}\phi Y=0$ and $\mathcal{T}%
_{V}\phi W=0$ for $W,$ $V\in \Gamma (\ker F_{\ast }),$ $X,$ $Y\in \Gamma
((\ker F_{\ast })^{\bot }).$
\end{theorem}

\subsection{\textbf{Anti-invariant submersions admitting horizontal
structure vector field }}

In this section, we will study anti-invariant submersions from a
cosymplectic manifold onto a Riemannian manifold such that the
characteristic vector field $\xi $ is horizontal. Using (\ref{A1}), we have $%
\mu =\phi \mu \oplus \{\xi \}.$ For any horizontal vector field $X$ we put 
\begin{equation}
\phi X=BX+CX,  \label{IREM}
\end{equation}%
where where $BX\in \Gamma (\ker F_{\ast })$ and $CX\in \Gamma (\mu ).$

Now we suppose that $V$ is vertical and $X$ is horizontal vector field.
Using above relation and (\ref{metric}) we obtain%
\begin{equation*}
g_{M}(\phi V,CX)=0.
\end{equation*}%
From this last relation we have $g_{N}(F_{\ast }\phi V,F_{\ast }CX)=0$ which
implies that 
\begin{equation}
TN=F_{\ast }(\phi (\ker F_{\ast }))\oplus F_{\ast }(\mu ).  \label{Ba}
\end{equation}

\begin{theorem}
Let $M(\phi ,\xi ,\eta ,g_{M})$ be a cosymplectic manifold \ of dimension $%
2m+1$ and $(N,g_{N})$ is a Riemannian manifold of dimension $n.$ Let $%
F:M(\phi ,\xi ,\eta ,g_{M})\rightarrow $ $(N,g_{N})$ be an anti-invariant
Riemannian submersion such that $(\ker F_{\ast })^{\bot }=\phi \ker F_{\ast
}\oplus \{\xi \}.$Then $m+1=n$.
\end{theorem}

\begin{proof}
We assume that $U_{1},...,U_{k\text{ }}$be an orthonormal frame of $\Gamma
(\ker F_{\ast })$, where $k=2m-n+1$. Since $(\ker F_{\ast })^{\bot }=\phi
\ker F_{\ast }\oplus \{\xi \}$, $\phi U_{1},...,\phi U_{k},\xi $ form an
orthonormal frame of $\Gamma ((\ker F_{\ast })^{\bot })$. So, by help of (%
\ref{A2a}) $\ $we obtain $k=n-1$ which implies that $m+1=n$.
\end{proof}

\begin{remark}
We note that Example 6 satisfies Theorem 10.
\end{remark}

From (\ref{fi}), (\ref{Ba}) and (\ref{IREM}) we obtain following Lemma.

\begin{lemma}
Let $F$ be an anti-invariant Riemannian submersion from a cosymplectic
manifold $M(\phi ,\xi ,\eta ,g_{M})$ to a Riemannian manifold $(N,g_{N})$.
Then we have%
\begin{eqnarray*}
BCX &=&0,\text{ \ } \\
C^{2}X &=&\phi ^{2}X-\phi BX, \\
C\phi V &=&0,\text{ \ }C^{3}X+CX=0, \\
B\phi V &=&-V
\end{eqnarray*}%
for any $X\in \Gamma ((\ker F_{\ast })^{\bot })$ and $V\in \Gamma ((\ker
F_{\ast }))$.
\end{lemma}

Using (\ref{nambla})\ one can easily obtain 
\begin{equation}
\nabla _{X}Y=-\phi \nabla _{X}\phi Y+\eta (\nabla _{X}Y)\xi
\label{Namblafi3}
\end{equation}%
for $X,Y\in \Gamma ((\ker F_{\ast })^{\bot }).$

\begin{lemma}
Let $F$ be an anti-invariant Riemannian submersion from a cosymplectic
manifold $M(\phi ,\xi ,\eta ,g_{M})$ to a Riemannian manifold $(N,g_{N})$.
Then we have%
\begin{equation}
\mathcal{A}_{X}\xi =0,  \label{IKE1}
\end{equation}%
\begin{equation}
\mathcal{T}_{U}\xi =0,  \label{IKE2}
\end{equation}%
\begin{equation}
g_{M}(CX,\phi U)=0,  \label{IKE3}
\end{equation}%
\begin{equation}
g_{M}(\nabla _{Y}CX,\phi U)=-g_{M}(CX,\phi \mathcal{A}_{Y}U),  \label{IKE4}
\end{equation}%
for $X,Y\in \Gamma ((\ker F_{\ast })^{\bot })$ and $U\in \Gamma (\ker
F_{\ast }).$
\end{lemma}

\begin{proof}
By virtue of (\ref{4}) and (\ref{xzeta}) we have (\ref{IKE1}). Using (\ref{2}%
) and (\ref{xzeta}) we obtain (\ref{IKE2}). For $X\in \Gamma ((\ker F_{\ast
})^{\bot })$ and $U\in \Gamma (\ker F_{\ast })$, by virtue of (\ref{metric})
and (\ref{IREM}) we get 
\begin{eqnarray}
g_{M}(CX,\phi U) &=&g_{M}(\phi X-BX,\phi U)  \label{6.5} \\
&=&g_{M}(X,U)-\eta (X)\eta (U)+g_{M}(\phi BX,U).  \notag
\end{eqnarray}%
Since $\phi BX\in \Gamma ((\ker F_{\ast })^{\bot })$ and $\xi \in \Gamma
(\ker F_{\ast }),$ (\ref{6.5}) implies (\ref{IKE3}). Now using (\ref{IKE3})
we get%
\begin{equation*}
g_{M}(\nabla _{Y}CX,\phi U)=-g_{M}(CX,\nabla _{Y}\phi U)
\end{equation*}%
for $X,Y\in \Gamma ((\ker F_{\ast })^{\bot })$ and $U\in \Gamma (\ker
F_{\ast })$. Then using (\ref{3}) and (\ref{nambla}) we have 
\begin{equation*}
g_{M}(\nabla _{Y}CX,\phi U)=-g_{M}(CX,\phi \mathcal{A}_{Y}U)-g_{M}(CX,\phi (%
\mathcal{V}\nabla _{Y}U)).
\end{equation*}%
Since $\phi (\mathcal{V}\nabla _{Y}U)\in \Gamma ((\ker F_{\ast })^{\bot }),$
we obtain (\ref{IKE4}).
\end{proof}

\begin{theorem}
Let $M(\phi ,\xi ,\eta ,g_{M})$ be a cosymplectic manifold\ of dimension $%
2m+1$ and $(N,g_{N})$ is a Riemannian manifold of dimension $n$. Let $%
F:M(\phi ,\xi ,\eta ,g_{M})\rightarrow $ $(N,g_{N})$ be an anti-invariant
Riemannian submersion. Then the fibers are not proper totally umbilical.
\end{theorem}

\begin{proof}
If the $(\ker F_{\ast })^{\bot }$ are proper totally umbilical, then we have 
\begin{equation}
g_{M}(\nabla _{X}Y,V)=g_{M}(A_{X}Y,V)=g_{M}(H,V)g_{M}(X,Y)  \label{TOT}
\end{equation}%
for $X,Y\in \Gamma ((\ker F_{\ast })^{\bot })$ and $V\in \Gamma (\ker
F_{\ast })$, where $H$ is the mean curvature vector field of $(\ker F_{\ast
})^{\bot }$. Putting $\xi $ instead of $Y$ in (\ref{TOT}) and then using (%
\ref{IKE1}) we get $H=0$. This shows that $(\ker F_{\ast })^{\bot }$ is
totally geodesic. This completes proof of the Theorem.
\end{proof}

We now study the integrability of the distribution $(\ker F_{\ast })^{\bot }$
and then we investigate the geometry of leaves of $\ker F_{\ast }$ and $%
(\ker F_{\ast })^{\bot }$.

\begin{theorem}
Let $F$ be an anti-invariant Riemannian submersion from a cosymplectic
manifold $M(\phi ,\xi ,\eta ,g_{M})$ to a Riemannian manifold $(N,g_{N})$.
Then the following assertions are equivalent to each other;
\end{theorem}

$\ i)$ $(\ker F_{\ast })^{\bot }$ \textit{is integrable.}

$\ ii)$ 
\begin{eqnarray*}
g_{N}((\nabla F_{\ast })(Y,BX),F_{\ast }\phi V) &=&g_{N}((\nabla F_{\ast
})(X,BX),F_{\ast }\phi V) \\
&&+g_{M}(CY,\phi \mathcal{A}_{X}V)-g_{M}(CX,\phi \mathcal{A}_{Y}V).
\end{eqnarray*}

$iii)$%
\begin{equation*}
g_{M}(\mathcal{A}_{X}BY-\mathcal{A}_{Y}BX,\phi V)=g_{M}(CY,\phi \mathcal{A}%
_{X}V)-g_{M}(CX,\phi \mathcal{A}_{Y}V)
\end{equation*}%
\textit{for }$X,Y\in \Gamma ((\ker F_{\ast })^{\bot })$\textit{\ and }$V\in
\Gamma (\ker F_{\ast }).$

\begin{proof}
Using (\ref{Namblafi3}), for $X,Y\in \Gamma ((\ker F_{\ast })^{\bot })$ and $%
V\in \Gamma (\ker F_{\ast })$ we get%
\begin{eqnarray*}
g_{M}(\left[ X,Y\right] ,V) &=&g_{M}(\nabla _{X}Y,V)-g_{M}(\nabla _{Y}X,V) \\
&=&g_{M}(\nabla _{X}\phi Y,\phi V)-g_{M}(\nabla _{Y}\phi X,\phi V).
\end{eqnarray*}%
Then from (\ref{IREM}) we have%
\begin{eqnarray*}
g_{M}(\left[ X,Y\right] ,V) &=&g_{M}(\nabla _{X}BY,\phi V)+g_{M}(\nabla
_{X}CY,\phi V)-g_{M}(\nabla _{Y}BX,\phi V) \\
&&-g_{M}(\nabla _{Y}CX,\phi V).
\end{eqnarray*}%
Using (\ref{AT2}), (\ref{3}) and (\ref{IKE4}) and if we take into account
that $F$ is a Riemannian submersion, we obtain%
\begin{eqnarray*}
g_{M}(\left[ X,Y\right] ,V) &=&g_{N}(F_{\ast }\nabla _{X}BY,F_{\ast }\phi
V)-g_{M}(CY,\phi \mathcal{A}_{X}V) \\
&&-g_{N}(F_{\ast }\nabla _{Y}BX,F_{\ast }\phi V)+g_{M}(CX,\phi \mathcal{A}%
_{Y}V).
\end{eqnarray*}%
Thus, from (\ref{5}) we have 
\begin{eqnarray*}
g_{M}(\left[ X,Y\right] ,V) &=&g_{N}(-(\nabla F_{\ast })(X,BY)+(\nabla
F_{\ast })(Y,BX),F_{\ast }\phi V) \\
&&+g_{M}(CX,\phi \mathcal{A}_{Y}V)-g_{M}(CY,\phi \mathcal{A}_{X}V)
\end{eqnarray*}%
which proves $(i)\Leftrightarrow (ii).$On the other hand using (\ref{5}) we
get%
\begin{equation*}
(\nabla F_{\ast })(Y,BX)-(\nabla F_{\ast })(X,BY)=-F_{\ast }(\nabla
_{Y}BX-\nabla _{X}BY).
\end{equation*}%
Then (\ref{3}) implies that 
\begin{equation*}
(\nabla F_{\ast })(Y,BX)-(\nabla F_{\ast })(X,BY)=-F_{\ast }(\mathcal{A}%
_{Y}BX-\mathcal{A}_{X}BY).
\end{equation*}

From (\ref{AT2}) $\mathcal{A}_{Y}BX-\mathcal{A}_{X}BY\in \Gamma ((\ker
F_{\ast })^{\bot }),$ this shows that $(ii)\Leftrightarrow (iii).$
\end{proof}

\begin{remark}
We assume that $(\ker F_{\ast })^{\bot }=\phi \ker F_{\ast }\oplus \{\xi \}.$
Using (\ref{IREM}) one can prove that $CX=0$ \textit{for }$X\in \Gamma
((\ker F_{\ast })^{\bot })$.
\end{remark}

\begin{corollary}
Let $M(\phi ,\xi ,\eta ,g_{M})$ be a cosymplectic manifold \ of dimension $%
2m+1$ and $(N,g_{N})$ is a Riemannian manifold of dimension $n.$ Let $%
F:M(\phi ,\xi ,\eta ,g_{M})\rightarrow $ $(N,g_{N})$ be an anti-invariant
Riemannian submersion such that $(\ker F_{\ast })^{\bot }=\phi \ker F_{\ast
}\oplus \{\xi \}.$ Then the following assertions are equivalent to each
other;

$i)$ $(\ker F_{\ast })^{\bot }$ \textit{is integrable.}

$ii)(\nabla F_{\ast })(X,\phi Y)=(\nabla F_{\ast })(\phi X,Y),$ for $X\in
\Gamma ((\ker F_{\ast })^{\bot })$ and $X,Y\in \Gamma ((\ker F_{\ast
})^{\bot }).$

$iii)$ $\mathcal{A}_{X}\phi Y=\mathcal{A}_{Y}\phi X.$
\end{corollary}

\begin{theorem}
Let $M(\phi ,\xi ,\eta ,g_{M})$ be a cosymplectic\ of dimension $2m+1$ and $%
(N,g_{N})$ is a Riemannian manifold of dimension $n.$ Let $F:M(\phi ,\xi
,\eta ,g_{M})\rightarrow $ $(N,g_{N})$ be an anti-invariant Riemannian
submersion. Then the following assertions are equivalent to each other;
\end{theorem}

$\ i)$ $(\ker F_{\ast })^{\bot }$ \textit{defines a totally geodesic
foliation on }$M.$

$\ ii)$ 
\begin{equation*}
g_{M}(\mathcal{A}_{X}BY,\phi V)=g_{M}(CY,\phi \mathcal{A}_{X}V).
\end{equation*}

$iii)$%
\begin{equation*}
g_{N}((\nabla F_{\ast })(X,\phi Y),F_{\ast }\phi V)=-g_{M}(CY,\phi \mathcal{A%
}_{X}V).
\end{equation*}%
\textit{for }$X,Y\in \Gamma ((\ker F_{\ast })^{\bot })$\textit{\ and }$V\in
\Gamma (\ker F_{\ast })$.

\begin{proof}
From (\ref{metric}) and (\ref{nambla}) we obtain%
\begin{equation*}
g_{M}(\nabla _{X}Y,V)=g_{M}(\nabla _{X}\phi Y,\phi V)
\end{equation*}%
for\textit{\ }$X,Y\in \Gamma ((\ker F_{\ast })^{\bot })$\textit{\ }and%
\textit{\ }$V\in \Gamma (\ker F_{\ast }).$Using (\ref{IREM})%
\begin{equation*}
g_{M}(\nabla _{X}Y,V)=g_{M}(\mathcal{\nabla }_{X}BY+\mathcal{\nabla }%
_{X}CY,\phi V)
\end{equation*}%
From (\ref{3}) and (\ref{C3}) 
\begin{equation*}
g_{M}(\nabla _{X}Y,V)=g_{M}(\mathcal{A}_{X}BY+\mathcal{V}\nabla _{X}BY,\phi
V)-g_{M}(CY,\phi \mathcal{A}_{X}V).
\end{equation*}

The last equation shows $(i)\Leftrightarrow (ii)$.

For $X,Y\in \Gamma ((\ker F_{\ast })^{\bot })$\textit{\ }and\textit{\ }$V\in
\Gamma (\ker F_{\ast }),$%
\begin{eqnarray}
g_{M}(\mathcal{A}_{X}BY,\phi V) &=&g_{M}(CY,\phi \mathcal{A}_{X}V)  \notag \\
&&\overset{(\ref{C3})}{=}-g_{M}(\nabla _{X}CY,\phi V)  \notag \\
&&\overset{(\ref{A2})}{=}-g_{M}(\nabla _{X}\phi Y,\phi V)+g_{M}(\nabla
_{X}BY,\phi V)  \label{K6}
\end{eqnarray}%
Since differential $F_{\ast }$ preserves the lenghts of horizontal vectors
the relation (\ref{K6}) forms%
\begin{equation}
g_{M}(\mathcal{A}_{X}BY,\phi V)=g_{N}(F_{\ast }\nabla _{X}BY,F_{\ast }\phi
V)-g_{M}(\nabla _{X}\phi Y,\phi V)  \label{K7}
\end{equation}%
Using, (\ref{nambla}), (\ref{metric}), (\ref{5}) and (\ref{5a}) in (\ref{K7}%
) respectively, we obtain%
\begin{equation*}
g_{M}(\mathcal{A}_{X}BY,\phi V)=g_{N}(-(\nabla F_{\ast })(X,\phi Y),F_{\ast
}\phi V)
\end{equation*}%
which tells that $(ii)\Leftrightarrow (iii)$.
\end{proof}

\begin{corollary}
Let $F:M(\phi ,\xi ,\eta ,g_{M})\rightarrow $ $(N,g_{N})$ be an
anti-invariant Riemannian submersion such that $(\ker F_{\ast })^{\bot
}=\phi \ker F_{\ast }\oplus \{\xi \}$, where $M(\phi ,\xi ,\eta ,g_{M})$ is
a cosymplectic manifold and $(N,g_{N})$ is a Riemannian manifold. Then the
following assertions are equivalent to each other;

$\ i)$ $(\ker F_{\ast })^{\bot }$ \textit{defines a totally geodesic
foliation on }$M.$

$ii)$ $\mathcal{A}_{X}\phi Y=0.$

$iii)$ $(\nabla F_{\ast })(X,\phi Y)=0$ \textit{for }$X,Y\in \Gamma ((\ker
F_{\ast })^{\bot })$\textit{\ and }$V\in \Gamma (\ker F_{\ast })$.
\end{corollary}

For the distribution $\ker F_{\ast }$, we have;

\begin{theorem}
Let $F$ be an anti-invariant Riemannian submersion from a cosymplectic
manifold $M(\phi ,\xi ,\eta ,g_{M})$ to a Riemannian manifold $(N,g_{N})$.
Then the following assertions are equivalent to each other;

$i)$ $(\ker F_{\ast })$ defines a totally geodesic foliation on $M$ .

$ii)$ $g_{N}((\nabla F_{\ast })(V,\phi X),F_{\ast }\phi W)=0$ \ \ for $X\in
\Gamma ((\ker F_{\ast })^{\bot })$ and $V,W\in \Gamma (\ker F_{\ast }).$

$iii)$ $\mathcal{T}_{V}BX+\mathcal{A}_{CX}V\in \Gamma (\mu ).$
\end{theorem}

\begin{proof}
Since $g_{M}(W,X)=0$ we have $g_{M}(\nabla _{V}W,X)=-g(W,\nabla _{V}X).$
From (\ref{metric}) and (\ref{nambla}) we get $g_{M}(\nabla
_{V}W,X)=-g_{M}(\phi W,H\nabla _{V}\phi X).$ Then Riemannian submersion $F$
and (\ref{5}) imply that 
\begin{equation*}
g_{M}(\nabla _{V}W,X)=g_{N}(F_{\ast }\phi W,(\nabla F_{\ast })(V,\phi X))
\end{equation*}%
which is $(i)\Leftrightarrow (ii).$ By direct calculation, we derive%
\begin{equation*}
g_{N}((F_{\ast }\phi W,(\nabla F_{\ast })(V,\phi X))=-g_{M}(\phi W,\nabla
_{V}\phi X).
\end{equation*}%
Using (\ref{IREM}) we have 
\begin{equation*}
g_{N}((F_{\ast }\phi W,(\nabla F_{\ast })(V,\phi X))=-g_{M}(\phi W,\nabla
_{V}BX+\nabla _{V}CX).
\end{equation*}%
Hence we get%
\begin{equation*}
g_{N}((F_{\ast }\phi W,(\nabla F_{\ast })(V,\phi X))=-g_{M}(\phi W,\nabla
_{V}BX+\left[ V,CX\right] +\nabla _{CX}V).
\end{equation*}%
Since $\left[ V,CX\right] \in \Gamma (\ker F_{\ast }),$ using (\ref{1}) and (%
\ref{3}), we obtain%
\begin{equation*}
g_{N}((F_{\ast }\phi W,(\nabla F_{\ast })(V,\phi X))=-g_{M}(\phi W,\mathcal{T%
}_{V}BX+\mathcal{A}_{CX}V).
\end{equation*}%
This shows $(ii)\Leftrightarrow (iii).$
\end{proof}

\begin{corollary}
Let $F:M(\phi ,\xi ,\eta ,g_{M})\rightarrow $ $(N,g_{N})$ be an
anti-invariant Riemannian submersion such that $(\ker F_{\ast })^{\bot
}=\phi \ker F_{\ast }\oplus \{\xi \},$ where $M(\phi ,\xi ,\eta ,g_{M})$ is
a cosymplectic manifold and $(N,g_{N})$ is a Riemannian manifold. Then
following assertions are equivalent to each other;

$i)(\ker F_{\ast })$ \textit{defines a totally geodesic foliation on }$M.$

$ii)(\nabla F_{\ast })(V,\phi X)=0,$ for $X\in \Gamma ((\ker F_{\ast
})^{\bot })$ and $V,W\in \Gamma (\ker F_{\ast }).$

$iii)$ $\mathcal{T}_{V}\phi W=0.$
\end{corollary}

\begin{theorem}
Let $F:M(\phi ,\xi ,\eta ,g_{M})\rightarrow $ $(N,g_{N})$ be an
anti-invariant Riemannian submersion such that $(\ker F_{\ast })^{\bot
}=\phi \ker F_{\ast }\oplus \{\xi \},$ where $M(\phi ,\xi ,\eta ,g_{M})$ is
a cosymplectic manifold and $(N,g_{N})$ is a Riemannian manifold. Then $F$
is a totally geodesic map if and only if%
\begin{equation}
\mathcal{T}_{W}\phi V=0,\text{ \ \ }\forall \text{ }W,\text{ }V\in \Gamma
(\ker F_{\ast })  \label{S1}
\end{equation}%
and%
\begin{equation}
\mathcal{A}_{X}\phi W=0,\text{ \ }\forall X\in \Gamma ((\ker F_{\ast
})^{\bot }),\forall \text{ }W\in \Gamma (\ker F_{\ast }).\text{\ }
\label{S2}
\end{equation}
\end{theorem}

\begin{proof}
First of all, we recall that the second fundamental form of a Riemannian
submersion satisfies (\ref{5a}). For $W,$ $V\in \Gamma (\ker F_{\ast }),$ by
using (\ref{2}), (\ref{5}), (\ref{fi}) and (\ref{nambla}), we get%
\begin{equation}
(\nabla F_{\ast })(W,V)=F_{\ast }(\phi \mathcal{T}_{W}\phi V).  \label{S3}
\end{equation}%
On the other hand by using (\ref{5}) and (\ref{nambla}) we have%
\begin{equation*}
(\nabla F_{\ast })(X,W)=F_{\ast }(\phi \nabla _{X}\phi W)
\end{equation*}%
for $X\in \Gamma ((\ker F_{\ast })^{\bot }).$ Then from (\ref{4}) and (\ref%
{fi}), we obtain 
\begin{equation}
(\nabla F_{\ast })(X,W)=F_{\ast }(\phi \mathcal{A}_{X}\phi W).  \label{S4}
\end{equation}%
Since $\phi $ is non-singular, using (\ref{4b}) and (\ref{4c}) proof comes
from (\ref{5a}), (\ref{S3}) and (\ref{S4}).
\end{proof}

Finally, we give a necessary and sufficient condition for an anti-invariant
Riemannian submersion such that $(\ker F_{\ast })^{\bot }=\phi \ker F_{\ast
}\oplus \{\xi \}$ to be harmonic.

\begin{theorem}
Let $F:M(\phi ,\xi ,\eta ,g_{M})\rightarrow $ $(N,g_{N})$ be an
anti-invariant Riemannian submersion such that $(\ker F_{\ast })^{\bot
}=\phi \ker F_{\ast }\oplus \{\xi \},$ where $M(\phi ,\xi ,\eta ,g_{M})$ is
a cosymplectic manifold and $(N,g_{N})$ is a Riemannian manifold. Then $F$
is harmonic if and only if Trace$\phi \mathcal{T}_{V}=0$ for $V\in \Gamma
(\ker F_{\ast }).$
\end{theorem}

\begin{proof}
From \cite{EJ} we know that $F$ is harmonic if and only if $F$ has minimal
fibres. Thus $F$ is harmonic if and only if $\sum\limits_{i=1}^{k}\mathcal{T}%
_{e_{i}}e_{i}=0,$ where $k$ is dimension of $\ker F_{\ast }$ . On the other
hand, from (\ref{1}), (\ref{2}) and (\ref{nambla}), we get%
\begin{equation}
\mathcal{T}_{V}\phi W=\phi \mathcal{T}_{V}W  \label{S5}
\end{equation}%
for any $W,$ $V\in \Gamma (\ker F_{\ast }).$ Using (\ref{S5}), we get 
\begin{equation*}
\sum\limits_{i=1}^{k}g_{M}(\mathcal{T}_{e_{i}}\phi
e_{i},V)=-\sum\limits_{i=1}^{k}g_{M}(\mathcal{T}_{e_{i}}e_{i},\phi V)
\end{equation*}%
for any $V\in \Gamma (\ker F_{\ast }).$ (\ref{4b}) implies that%
\begin{equation*}
\sum\limits_{i=1}^{k}g_{M}(\phi e_{i},\mathcal{T}_{e_{i}}V)=\sum%
\limits_{i=1}^{k}g_{M}(\mathcal{T}_{e_{i}}e_{i},\phi V).
\end{equation*}%
Then, using (\ref{TUW}) we have%
\begin{equation*}
\sum\limits_{i=1}^{k}g_{M}(\phi e_{i},\mathcal{T}_{Ve_{i}})=\sum%
\limits_{i=1}^{k}g_{M}(\mathcal{T}_{e_{i}}e_{i},\phi V).
\end{equation*}%
Hence, proof comes from (\ref{metric}).
\end{proof}

From Theorem 13 and Theorem 14 we have following Theorem.

\begin{theorem}
Let $F$ be an anti-invariant Riemannian submersion from a cosymplectic
manifold $M(\phi ,\xi ,\eta ,g_{M})$ to a Riemannian manifold $(N,g_{N}).$
Then $M$ is a locally product manifold if and only if%
\begin{equation*}
g_{N}((\nabla F_{\ast })(X,\phi Y),F_{\ast }\phi V)=-g_{M}(CY,\phi \mathcal{A%
}_{X}V)
\end{equation*}%
and%
\begin{equation*}
g_{N}((\nabla F_{\ast })(V,\phi X),F_{\ast }\phi W)=0
\end{equation*}%
for $W,$ $V\in \Gamma (\ker F_{\ast }),$ $X,$ $Y\in \Gamma ((\ker F_{\ast
})^{\bot }).$
\end{theorem}

Using Corollary 5 and Corollary 6, we get following Theorem.

\begin{theorem}
Let $F:M(\phi ,\xi ,\eta ,g_{M})\rightarrow $ $(N,g_{N})$ be an
anti-invariant Riemannian submersion such that $(\ker F_{\ast })^{\bot
}=\phi \ker F_{\ast }\oplus \{\xi \},$ where $M(\phi ,\xi ,\eta ,g_{M})$ is
a cosymplectic manifold and $(N,g_{N})$ is a Riemannian manifold. Then $M$
is a locally product manifold if and only if \ $\mathcal{A}_{X}\phi Y=0$ and 
$\mathcal{T}_{V}\phi W=0$ for $W,$ $V\in \Gamma (\ker F_{\ast }),$ $X,$ $%
Y\in \Gamma ((\ker F_{\ast })^{\bot }).$
\end{theorem}

\end{document}